\documentclass[12pt,a4paper]{amsart}
\usepackage{amsfonts}
\usepackage{amsthm}
\usepackage{amsmath}
\usepackage{amscd}
\usepackage[latin2]{inputenc}
\usepackage{t1enc}
\usepackage[mathscr]{eucal}
\usepackage{indentfirst}
\usepackage{graphicx}
\usepackage{graphics}
\usepackage{pict2e}
\usepackage{epic}
\usepackage{url}

\usepackage{cite}
\usepackage{color}
\usepackage{epic}
\numberwithin{equation}{section}
\usepackage[margin=2.9cm]{geometry}
\usepackage{cite}
\usepackage{color}
\usepackage{epic}

\theoremstyle{plain}
\newtheorem{Th}{Theorem}[section]
\newtheorem{Lemma}[Th]{Lemma}
\newtheorem{Cor}[Th]{Corollary}

 \theoremstyle{definition}

\newtheorem{Rem}[Th]{Remark}
\newtheorem{?}[Th]{Problem}

\newcommand{\Hom}{{\rm{Hom}}}

\renewcommand{\wr}{\mathrm{wr}}
\newcommand{\ch}{\mathrm{ch}}

\begin{document}

\title[Colorings]{Sidorenko's conjecture, colorings \\and independent sets}

\author[P. Csikv\'ari]{P\'{e}ter Csikv\'{a}ri}

\address[P\'{e}ter Csikv\'{a}ri]{Massachusetts Institute of Technology \\ Department of Mathematics \\
Cambridge MA 02139 \& MTA-ELTE Geometric and Algebraic Combinatorics Research Group
\\ H-1117 Budapest
\\ P\'{a}zm\'{a}ny P\'{e}ter s\'{e}t\'{a}ny 1/C \\ Hungary}

\email{peter.csikvari@gmail.com}

\author[Z. Lin]{Zhicong Lin}
\address[Zhicong Lin]{School of Science, Jimei University, Xiamen 361021,
P.R. China
\& CAMP, National Institute for Mathematical Sciences, Daejeon 34047, Republic of Korea} 
\email{lin@nims.re.kr}

\thanks{The first author  is partially supported by the National Science Foundation under grant no. DMS-1500219, and by the MTA R\'enyi "Lend\"ulet" Groups and Graphs Research Group, and by the ERC Consolidator Grant 648017, and by the Hungarian National Research, Development and Innovation Office, NKFIH grants K109684 and NN114614, a Slovenian-Hungarian grant. The second author is supported by the National Science Foundation of China grant no.~11501244.}



\begin{abstract} Let $\hom(H,G)$ denote the number of homomorphisms from a graph $H$ to a graph $G$. Sidorenko's conjecture asserts that for any bipartite graph $H$, and a graph $G$ we have
$$\hom(H,G)\geq v(G)^{v(H)}\left(\frac{\hom(K_2,G)}{v(G)^2}\right)^{e(H)},$$
where $v(H),v(G)$ and $e(H),e(G)$ denote the number of vertices and edges of the graph $H$ and $G$, respectively.
In this paper we prove Sidorenko's conjecture for certain special graphs $G$: for the complete graph $K_q$ on $q$ vertices, for a $K_2$ with a loop added at one of the end vertices, and for a path on $3$ vertices with a loop added at each vertex. These cases correspond to counting colorings, independent sets and Widom-Rowlinson configurations of a graph $H$. For instance, for a bipartite graph $H$ the number of $q$-colorings $\ch(H,q)$ satisfies
$$\ch(H,q)\geq q^{v(H)}\left(\frac{q-1}{q}\right)^{e(H)}.$$
In fact, we will prove that in the last two cases (independent sets and Widom-Rowlinson configurations) the graph $H$ does not need to be bipartite.
In all cases, we first prove a certain correlation inequality which implies Sidorenko's conjecture in a stronger form.

\end{abstract}

\maketitle

\section{Introduction} Let $v(G)$ and $e(G)$ denote the number of vertices and edges of a graph $G$. For a graph $H$ and $G$ let $\hom(H,G)$ denote the number of {\em homomorphisms} from $H$ to $G$, i.e., the number of maps $\varphi:V(H)\to V(G)$ such that $(\varphi(u),\varphi(v))\in E(G)$ whenever $(u,v)\in E(H)$. 

Sidorenko's conjecture \cite{S1,S2} claims that for a bipartite graph $H$ on $v(H)=n$ vertices and $e(H)$ edges, and an arbitrary graph $G$ (possibly with loops) we have
$$\hom(H,G)\geq v(G)^{n}\left(\frac{\hom(K_2,G)}{v(G)^2}\right)^{e(H)}.$$

There has been many work on Sidorenko's conjecture, see for instance \cite{BR,CFS,CL,CKLL,H,KLL,LS,L,P,S1,S2,Sz1,Sz2}. It is known that Sidorenko's conjecture holds true if $H$ belongs to certain classes of graphs, for instance if $H$ is a tree or a complete bipartite graph. In this paper we follow another route: we prove Sidorenko's conjecture when $G$ is fixed, namely when $G$ is a complete graph (counting colorings) or a complete graph on $2$ vertices with a loop added at one of the vertices (counting independent sets), or a path on $3$ vertices with a loop added at each vertex (counting Widom-Rowlinson configurations). The idea of fixing the target graph is not new, in the paper \cite{L} Lov\'asz also follows the same route. Lov\'asz proves  the continuous version of Sidorenko's conjecture when the target graphon (the continuous version of a graph) is sufficiently close to the constant graphon in some metric, see Theorem 4.1 and 5.1 in \cite{L}. We remark that it is slightly inconvenient in these theorems that the distance between the target graphon and the constant graphon depends on the number of edges of the bipartite graph $H$. We will not assume such condition in our results.
\bigskip

When $G=K_q$, the complete graph on $q$ vertices, then $\hom(H,K_q)=\ch(H,q)$ counts the proper colorings of $H$ with $q$ colors. In this case Sidorenko's conjecture states that
$$\ch(H,q)\geq q^n\left(\frac{q-1}{q}\right)^{e(H)}.$$

When $G$ is a $K_2$ with a loop added at one vertex then $\hom(H,G)=i(H)$ counts the number of independent sets of the graph $H$. In this case Sidorenko's conjecture states that
$$i(H)\geq 2^n\left(\frac{3}{4}\right)^{e(H)}.$$

Note that both colorings and independent sets are related to certain statistical physical models, the {\em Potts model} and the {\em hard-core model}. Another notable statistical physical model is the so-called {\em Widom-Rowlinson model}~\cite{WR,BHW}. This model corresponds to the case when $G$ is a path on $3$ vertices with a loop added at each vertex (see Fig.~\ref{widom}). We denote this graph by $P_3^{\circ}$, then $\hom(H,P_3^{\circ})=\wr(H)$ counts the number of $3$--colorings of $H$ with colors white, red and blue such that a red and a blue vertex cannot be adjacent, but there is no other restriction. Indeed, for a homomorphism $\phi:V(H)\to P_3^{\circ}$ let us consider the pre-image sets $A=\phi^{-1}(\{a\})$ (red), $B=\phi^{-1}(\{b\})$ (white) and $C=\phi^{-1}(\{c\})$ (blue) then the only restriction is that a red and a blue vertex cannot be adjacent. This model has also been extensively studied in extremal graph theory, see the recent work by Cutler and Radcliffe~\cite{CR} and by Cohen, Perkins and Tetali~\cite{CPT}.
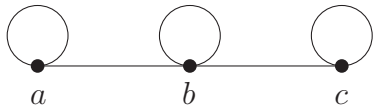
\begin{figure}[h]
\setlength {\unitlength} {0.8mm}
\begin {picture} (65,25) \setlength {\unitlength} {1mm}
\thinlines
\put(5,7){\circle*{1.8}}
\put(5,11){\circle{8}} \put(4,2){$a$}
\put(5,7){\line(1,0){20}}

\put(25,7){\circle*{1.8}}
\put(25,11){\circle{8}} \put(24,2){$b$}
\put(25,7){\line(1,0){20}}

\put(45,7){\circle*{1.8}}\put(44,2){$c$}
\put(45,11){\circle{8}}

\end{picture}

\caption{\label{widom}The graph $P_3^{\circ}$}
\end {figure}

In the case of Widom-Rowlinson configuration Sidorenko's conjecture states that
$$\wr(H)\geq 3^n\left(\frac{7}{9}\right)^{e(H)}$$
as $v(P_3^{\circ})=3$, and $\hom(K_2,P_3^{\circ})=7$.

In this short note we prove all these claims. In fact, in all cases we prove a correlation inequality from which the claim follows in a stronger form.

\begin{Th} \label{coloring_bipartite} Let $H=(A,B,E)$ be a bipartite graph. Let $c$ be a uniform random proper coloring of $H$ with $q$ colors. Then,
$$\mathbb{P}(c(u)=c(v))\leq \frac{1}{q}$$
if $u$ and $v$ are in different parts of the bipartite graph, and
$$\mathbb{P}(c(u)=c(v))\geq \frac{1}{q}$$
if $u$ and $v$ are in the same part of the bipartite graph. The former part is equivalent with the statement that for any bipartite graph $H$ with a missing edge $e$ between the two parts we have
\begin{equation}\label{eq:col}
\frac{\ch(H+e,q)}{\ch(H,q)}\geq \frac{q-1}{q}.
\end{equation}
\end{Th}

\begin{Cor} \label{count_coloring} For any bipartite graph $H$ on $n$ vertices and $e(H)$ edges we have 
$$\ch(H,q)\geq q^{n}\left(\frac{q-1}{q}\right)^{e(H)}.$$
Furthermore, if the graph $H$ contains $\varepsilon n$ vertex disjoint cycles of length at most $\ell$, then there is some $c_q(\varepsilon, \ell)>1$ such that
$$\ch(H,q)\geq c_q(\varepsilon, \ell)^nq^{n}\left(\frac{q-1}{q}\right)^{e(H)}.$$
\end{Cor}

We will see that the first inequality in Corollary~\ref{count_coloring} is asymptotically tight for sparse graphs not containing short cycles if $q$ is (much) larger than the largest degree of the graph $H$, see Remark~\ref{rem}.
\bigskip

We will prove a very similar result for the number of independent sets. In fact, in the case of independent sets we can drop the condition of bipartiteness of $H$. On the other hand, this result is only asymptotically tight in very trivial cases unlike the bound on the number of colorings, but at least it shows that Sidorenko's conjecture holds true in this case. This result is much simpler than the claim on colorings.

\begin{Th} \label{independent}  For any graph $H$ and its edge $e$ we have 
\begin{equation}\label{eq:ind}
\frac{i(H)}{i(H-e)}\geq \frac{3}{4}.
\end{equation}
\end{Th}

\begin{Cor} \label{independent_count} For any graph $H$ on $n$ vertices and $e(H)$ edges we have
$$i(H)\geq 2^n\left(\frac{3}{4}\right)^{e(H)}.$$
Furthermore, if the graph $H$ is connected then,
$$i(H)> \frac{3}{4} \left(\frac{1+\sqrt{5}}{3}\right)^n\cdot 2^n\left(\frac{3}{4}\right)^{e(H)}.$$
\end{Cor}

Finally, we prove an analogous theorem for Widom-Rowlinson configurations. 

\begin{Th}\label{wr-model}
For any graph $H$ and its edge $e$ we have
\begin{equation}\label{eq:wr}
\frac{\wr(H)}{\wr(H-e)}\geq\frac{7}{9}.
\end{equation}
\end{Th}

Then, as for the independent set case, the above result implies the following extended version of Sidorenko's  conjecture for the graph $P_3^{\circ}$.

\begin{Cor}\label{sido:wr}
For any graph $H$ on $n$ vertices and $e(H)$ edges we have
$$\wr(H)\geq 3^n\left(\frac{7}{9}\right)^{e(H)}.$$
Furthermore, if the graph $H$ is connected then,
$$\wr(H)>\frac{9}{10}\left(\frac{3(1+\sqrt{2})}{7}\right)^n\cdot 3^n\left(\frac{7}{9}\right)^{e(H)}.$$
\end{Cor}

\bigskip

\noindent \textbf{Notation.} In this paper all graphs are simple unless otherwise specified. For a graph $H=(V(H),E(H))$ and an edge set $E'\subseteq E(H)$, the graph $H-E'$ is the graph $(V(H),E(H)\setminus E')$. Similarly, if $V'\subseteq V$ then $H-V'$ denotes the subgraph of $H$ induced by the vertex set $V(H)\setminus V'$. 
\bigskip

\noindent \textbf{This paper is organized as follows.} In the next section we study colorings of bipartite graphs, we will prove Theorem~\ref{coloring_bipartite} and Corollary~\ref{count_coloring}. In Section 3 we study the number of independent sets, we will prove Theorem~\ref{independent} and Corollary~\ref{independent_count}. In Section 4 we study the number of Widom-Rowlinson configurations, we will prove Theorem~\ref{wr-model} and Corollary~\ref{sido:wr}. Finally, we conclude the paper with an interesting open problem. 

\section{Colorings}

In this section we prove Theorem~\ref{coloring_bipartite}. 

\begin{proof}[Proof of Theorem~\ref{coloring_bipartite}]
Note that by the symmetry of colors we have 
$$\mathbb{P}(c(u)=c(v)=1)=\frac{1}{q}\mathbb{P}(c(u)=c(v)),$$
and
$$\mathbb{P}(c(u)=1, c(v)=2)=\frac{1}{q(q-1)}\mathbb{P}(c(u)\neq c(v)).$$
Hence the claim 
$$\mathbb{P}(c(u)=c(v))\leq \frac{1}{q}$$
is equivalent with 
$$\mathbb{P}(c(u)=c(v)=1)\leq \mathbb{P}(c(u)=1, c(v)=2)$$
if $u$ and $v$ are in different parts of the bipartite graph. Equivalently,
$$\mathbb{P}(c(u)=c(v)=1)+\mathbb{P}(c(u)=c(v)=2) \leq \mathbb{P}(c(u)=1, c(v)=2)+\mathbb{P}(c(u)=2, c(v)=1).$$
Now let us conditioning on every colors different from colors $1$ and $2$. So assume that we know the colors of every vertex except the colors of those who gets color $1$ or $2$. Let us consider the partial colorings $c$ in which the color of a vertex is fixed if it is not $1$ or $2$, and the uncolored vertices will be assigned color $1$ or $2$. Note that when we color the uncolored vertices by $1$ and $2$ we only have to make sure that the components of the induced subgraph of uncolored vertices must be properly colored. Now if $u$ and $v$ are in different components then by flipping the coloring on the component of $u$ we see that the color of $u$ and $v$ will be the same or different in the same number of colorings which extends the partial coloring. On the other hand, if they are in the same component then $u$ and $v$ must have different color. Hence the claim follows.

The claim when $u$ and $v$  are in the same part works exactly the same way, only this time if $u$ and $v$ are in the same component then their color is the same.
\bigskip

Now if $H$ is a bipartite graph with a missing edge $e=(u,v)\in E(H)$ between the two parts. Then the probability that in a proper coloring of $H$, the vertices $u$ and $v$ will get different colors is 
$$\frac{\ch(H+e,q)}{\ch(H,q)}\geq 1-\frac{1}{q}\geq \frac{q-1}{q}.$$
So this is just rephrasing the first claim.
\end{proof}

\begin{proof}[Proof of Corollary~\ref{count_coloring}.] Let $e(H)=m$, and let $E(H)=\{e_1,e_2,\dots ,e_m\}$ then by the last claim we have
$$\frac{\ch(H,q)}{\ch(H-\{e_1,\dots,e_{m}\},q)}=\frac{\ch(H,q)}{\ch(H-e_1,q)}\frac{\ch(H-e_1,q)}{\ch(H-\{e_1,e_2\},q)}\dots \frac{\ch(H-\{e_1,\dots,e_{m-1}\},q)}{\ch(H-\{e_1,e_2,\dots ,e_m\},q)}\geq $$
$$\geq \left(\frac{q-1}{q}\right)^m.$$
Note that $H-\{e_1,e_2,\dots ,e_m\}$ is simply the empty graph on $|V(H)|=n$ vertices so $\ch(H-\{e_1,e_2,\dots ,e_m\},q)=q^n$.
Hence
$$\ch(H,q)\geq q^n\left(\frac{q-1}{q}\right)^m.$$

To prove the second statement, let $S$ be the union of $\varepsilon n$ vertex disjoint cycles of length at most $\ell$ together with the remaining isolated vertices. In other words, $S=C_1\cup C_2\cup \dots \cup C_k\cup (n-e(S))K_1$, where $C_i$ is a cycle of even length $t_i\leq \ell$, and $k=\varepsilon n$. Let $E(H)\setminus E(S)=\{e_1,e_2,\dots ,e_r\}$, so $r=e(H)-e(S)$. Then,
$$\frac{\ch(H,q)}{\ch(S,q)}=\frac{\ch(H,q)}{\ch(H-e_1,q)}\frac{\ch(H-e_1,q)}{\ch(H-\{e_1,e_2\},q)}\dots \frac{\ch(H-\{e_1,\dots,e_{r-1}\},q)}{\ch(H-\{e_1,e_2,\dots ,e_r\},q)}\geq $$
$$\geq \left(\frac{q-1}{q}\right)^r.$$
We will need the fact that if $C_{\ell}$ is a cycle of length $\ell$ then 
$$\ch(C_{\ell},q)=(q-1)^{\ell}+(-1)^{\ell}(q-1).$$
This can be seen by induction on $\ell$ by using the recursion 
$$\ch(G,q)=\ch(G-e,q)-\ch(G/e,q),\quad\text{for any $e\in E(G)$.}$$
 Alternatively, if the Reader is familiar with spectral graph theory then we can use the observation that $\ch(C_{\ell},q)=\hom(C_{\ell},K_q)$. In general, $\hom(C_{\ell},G)$ counts the number of closed walks of length $\ell$ in a graph $G$. This is also $\sum_{i=1}^n\lambda_i^{\ell}$, where $\lambda_i$'s are the eigenvalues of the graph $G$. The eigenvalues of $K_q$ is $q-1$ and $-1$ with multiplicity $q-1$. This gives that $\hom(C_{\ell},K_q)=(q-1)^{\ell}+(q-1)(-1)^{\ell}$. In our case, the cycles have even lengths so we can omit the term $(-1)^{\ell}$.
	
Note that $e(S)=|C_1|+\dots +|C_k|$, and as we have seen
$$\ch(C_i,q)=(q-1)^{t_i}+(q-1)=(q-1)^{t_i}\left(1+\frac{1}{(q-1)^{t_i-1}}\right)\geq (q-1)^{t_i}\left(1+\frac{1}{(q-1)^{\ell-1}}\right).$$

Hence
$$\ch(S,q)\geq q^{n-e(S)}(q-1)^{e(S)}\left(1+\frac{1}{(q-1)^{\ell-1}}\right)^k.$$
Then
$$\ch(H,q)\geq ch(S,q)\left(\frac{q-1}{q}\right)^r\geq q^{n-e(S)}(q-1)^{e(S)}\left(1+\frac{1}{(q-1)^{\ell-1}}\right)^{\varepsilon n}\left(\frac{q-1}{q}\right)^r.$$
Hence
$$\ch(H,q)\geq c_q(\varepsilon, \ell)^nq^n\left(\frac{q-1}{q}\right)^{e(H)},$$ 
where
$$c_q(\varepsilon, \ell)=\left(1+\frac{1}{(q-1)^{\ell-1}}\right)^{\varepsilon}.$$
\end{proof}

\begin{Rem} \label{rem} So we see that for a bipartite graph $H$ on $v(H)$ vertices and $e(H)$ edges we have 
$$\frac{\ln \ch(H,q)}{v(H)}\geq \ln(q)+\frac{e(H)}{v(H)}\ln \left(\frac{q-1}{q}\right).$$
It is known that it is asymptotically tight in the following sense: let $G$ be a graph with girth $g$ (i.e., the length of the shortest cycle is $g$), and assume that $q>8d$, where $d$ is the largest degree of the graph $G$. Then Ab\'ert and Hubai \cite{AH} proved that
$$\left|\frac{\ln \ch(G,q)}{v(G)}-\left( \ln(q)+\frac{e(G)}{v(G)}\ln \left(\frac{q-1}{q}\right)\right)\right|\leq 2\frac{(8d/q)^{g-1}}{1-8d/q}.$$
This is true for any graph $G$, not just bipartite. The constant $8$ comes from a theorem of Sokal \cite{So} which asserts that the chromatic polynomial of a graph $G$ has no zero of absolute value bigger than $Cd$, where $C$ is a constant less than $8$. Sokal \cite{So} actually proved $C$ to be $7.96...$, which was improved to $6.91...$ by Fernandez and Procacci \cite{FP}. It is conjectured that $C$ can be taken to be $2$.

In the particular case when $G$ is $d$-regular, it was proved by Bandyopadhyay and Gamarnik \cite{BG} that for $q\geq d+1$ we have
$$\lim_{g \to \infty}\sup_{G\in \mathcal{G}(n,g,d)}\left|\frac{\ln \ch(G,q)}{n}-\left( \ln(q)+\frac{d}{2}\ln \left(\frac{q-1}{q}\right)\right)\right|=0,$$
where $\mathcal{G}(n,g,d)$ is the set of $d$--regular graphs on $n$ vertices with girth at least $g$.
So in the special case of regular graphs we can improve the constant $8$ to $1$. This theorem shows that even for regular graphs our inequality is asymptotically tight if $q\geq d+1$. In other words, for $q\geq d+1$ we have
$$\inf_{G\in \mathcal{G}^b_d} \ch(G,q)^{1/v(G)}=q\left(\frac{q-1}{q}\right)^{d/2},$$
where $\mathcal{G}^b_d$ denotes the family of $d$--regular bipartite graphs. If $d\geq 2$ the infimum is not achieved by a finite graph.

Another related result is about colorings of random graphs. Let $G_n=G(n,\frac{c}{n})$ be an Erd\H os-R\'enyi random graph. Then it is known \cite{CJ,CDGS} that for a fixed $q$ and for small $c$ we have
$$\lim_{n\to \infty}\mathbb{E}\frac{\ln \ch(G_n,q)}{v(G_n)}=\ln(q)+\frac{c}{2}\ln \left(\frac{q-1}{q}\right).$$
Interestingly, there is a phase transition at some $c_q$, where the left hand side starts to be strictly smaller than the right hand side, see for instance \cite{CDGS,Z} for details. Now we see that for bipartite Erd\H os-R\'enyi random graphs this cannot happen. On the other hand, if $H=(A,B,E)$ is a balanced bipartite graph, i.e., $|A|=|B|$ then even for a very dense bipartite graph we have
$$\ch(H,q)\geq \left\{\begin{array}{cl} (q/2)^n &\, \mbox{if}\, \, \, q\, \, \mbox{is even}, \\
((q-1)(q+1)/4)^{n/2} &\, \mbox{if}\, \, \, q\, \, \mbox{is odd}. \end{array} \right. $$
So if we introduce the function
$$f(q)=\left\{\begin{array}{cl} \ln\left(\frac{q}{2}\right) &\, \mbox{if}\, \, \, q\, \, \mbox{is even}, \\
\frac{1}{2}\ln\left( \frac{(q-1)(q+1)}{2}\right) &\, \mbox{if}\, \, \, q\, \, \mbox{is odd}; \end{array} \right. $$
then it is a natural question whether for balanced bipartite Erd\H os-R\'enyi random graphs we have
$$\lim_{n\to \infty \atop 2\ |\ n}\mathbb{E}\frac{\ln \ch(G_n,q)}{v(G_n)}=\max \left(\ln(q)+\frac{c}{2}\ln \left(\frac{q-1}{q}\right),f(q)\right).$$
\end{Rem}

\section{Number of independent sets}

In this section we prove Theorem~\ref{independent} and Corollary~\ref{independent_count}.

\begin{proof}[Proof of Theorem~\ref{independent}.]
Let $i(H \ |\ \textrm{condition})$ be the number of independent sets of $H$ satisfying the given condition. For $e=(u,v)\in E(H)$ we have
$$i(H)=i(H\ |\ u,v\notin I)+i(H\ |\ u\in I,v\notin I)+i(H\ |\ u\notin I, v\in H)+i(H\ |\ u,v\in I).$$
Note that
$$i(H\ |\ u,v\notin I)=i(H-\{u,v\}),\quad i(H\ |\ u\in I,v\notin I)=i(H-\{u,v\}-N(u))$$
and
$$i(H\ |\ u\notin I, v\in H)=i(H-\{u,v\}-N(v)), \, \, i(H\ |\ u,v\in I)=0.$$
Similarly,
$$i(H-e)=i(H-e\ |\ u,v\notin I)+i(H-e\ |\ u\in I,v\notin I)+i(H-e\ |\ u\notin I, v\in H)+i(H-e\ |\ u,v\in I).$$
Note that
$$i(H-e\ |\ u,v\notin I)=i(H-\{u,v\}),\quad i(H-e\ |\ u\in I,v\notin I)=i(H-\{u,v\}-N(u)),$$ 
$$i(H-e\ |\ u\notin I,v\in I)=i(H-\{u,v\}-N(v)), 
$$
and
$$ i(H-e\ |\ u,v\in I)=i(H-\{u,v\}-N(u)-N(v)).$$
Clearly,
$$i(H-\{u,v\}-N(u)-N(v))\leq \min(i(H-\{u,v\}),i(H-\{u,v\}-N(u)),i(H-\{u,v\}-N(v)),$$
whence inequality~\eqref{eq:ind} follows. 
%$$i(H)\geq \frac{3}{4}i(H-e).$$
\end{proof}

\begin{proof}[Proof of Corollary~\ref{independent_count}.]
The proof of this statement follows the same way from Theorem~\ref{independent} as
Corollary~\ref{count_coloring} followed from Theorem~\ref{coloring_bipartite}. Indeed,
$$i(H)\geq 2^n\left(\frac{3}{4}\right)^{e(H)}$$
immediately follows from the fact that the empty graph on $n$ vertices has $2^n$ independent sets and inequality~\eqref{eq:ind}.
%$$i(H)\geq \frac{3}{4}i(H-e).$$
%%\bigskip

If $H$ is connected then let $T_n$ be a spanning tree of $H$. Similarly, inequality~\eqref{eq:ind} implies
$$i(H)\geq i(T_n)\left(\frac{3}{4}\right)^{e(H)-(n-1)}.$$
Note that among trees on $n$ vertices the path $P_n$ has the smallest number of independent sets, for a proof see for instance \cite{PT,C}. Hence
$$i(H)\geq i(P_n)\left(\frac{3}{4}\right)^{e(H)-(n-1)}.$$
It is easy to see that $i(P_n)=i(P_{n-1})+i(P_{n-2})$ and $i(P_1)=2,i(P_2)=3$, whence $i(P_n)=F_{n+2}$, where $F_0=0,F_1=1$ and $F_n=F_{n-1}+F_{n-2}$ if $n\geq 2$, the sequence of {\em Fibonacci-numbers}. It is easy to see by induction (or by the explicit formula for $F_n$) that
$$i(P_n)>\left(\frac{1+\sqrt{5}}{2}\right)^n.$$
Hence
$$i(H)>\left(\frac{1+\sqrt{5}}{2}\right)^n\left(\frac{3}{4}\right)^{e(H)-(n-1)}=\frac{3}{4} \left(\frac{1+\sqrt{5}}{3}\right)^n\cdot 2^n\left(\frac{3}{4}\right)^{e(H)}.$$
\end{proof}

\begin{Rem} For bipartite graphs a much better lower bound follows from the work of N. Ruozzi \cite{R}. This lower bound connects the number of independent sets with the so-called Bethe partition function of the hard-core model.

\end{Rem}

\section{Widom-Rowlinson model}

The purpose of this section is to prove Theorem~\ref{wr-model}, we will use a similar approach as in the proof of Theorem~\ref{independent} with an extra crucial lemma.

\begin{proof}[Proof of Theorem~\ref{wr-model}]
For each $\phi\in\Hom(H,P_3^{\circ})$ and a vertex $u$ of $H$ let us write $u\in A,B$ or $C$ if and only if $\phi(u)=a,b$ or $c$, respectively. Then, for each $e=(u,v)\in E(H)$ we have
\begin{align*}
\wr(H)&=\wr(H\,\,|\,\,u,v\in A)+\wr(H\,\,|\,\,u\in A,v\in B)+\wr(H\,\,|\,\,u\in A, v\in C)\\
&\quad+\wr(H\,\,|\,\,u\in B, v\in A)+\wr(H\,\,|\,\,u,v\in B)+\wr(H\,\,|\,\,u\in B, v\in C)\\
&\quad+\wr(H\,\,|\,\,u\in C, v\in A)+\wr(H\,\,|\,\,u\in C,v\in B)+\wr(H\,\,|\,\,u, v\in C).
\end{align*}
In view of the symmetry of $P_3^{\circ}$, the above expression reduces to
\begin{align*}
\wr(H)&=2\wr(H\,\,|\,\,u,v\in A)+2\wr(H\,\,|\,\,u\in A,v\in B)+2\wr(H\,\,|\,\,u\in A, v\in C)\\
&\quad+2\wr(H\,\,|\,\,u\in B, v\in A)+\wr(H\,\,|\,\,u,v\in B)\\
&=2\biggl(\wr(H\,\,|\,\,u,v\in A)+\wr(H\,\,|\,\,u\in A,v\in B)+\wr(H\,\,|\,\,u\in B, v\in A)\biggr)\\
&\quad+\wr(H\,\,|\,\,u,v\in B),
\end{align*}
 since $(u,v)$ is an edge of $H$ and so $\wr(H\,\,|\,\,u\in A, v\in C)$ must vanish. On the other hand, we have
\begin{align*}
\wr(H-e)&=2\wr(H-e\,\,|\,\,u,v\in A)+2\wr(H-e\,\,|\,\,u\in A,v\in B)+2\wr(H-e\,\,|\,\,u\in A, v\in C)\\
&\quad+2\wr(H-e\,\,|\,\,u\in B, v\in A)+\wr(H-e\,\,|\,\,u,v\in B).
\end{align*}
It is easy to see that 
\begin{align*}
\wr(H\,\,|\,\,u,v\in A)&=\wr(H-e\,\,|\,\,u,v\in A),\\
\wr(H\,\,|\,\,u\in A,v\in B)&=\wr(H-e\,\,|\,\,u\in A,v\in B),\\
\wr(H\,\,|\,\,u\in B, v\in A)&=\wr(H-e\,\,|\,\,u\in B,v\in A),\\
\wr(H\,\,|\,\,u,v\in B)&=\wr(H-e\,\,|\,\,u,v\in B)
\end{align*}
and
$$
\wr(H\,\,|\,\,u,v\in A)\leq\min\{\wr(H\,\,|\,\,u\in A,v\in B),\wr(H\,\,|\,\,u\in B, v\in A),\wr(H\,\,|\,\,u,v\in B)\}.
$$
 These inequalities imply the desired inequality~\eqref{eq:wr} if we take the following lemma into account.
 
\begin{Lemma}For each $e=(u,v)\in E(H)$ we have
$$\wr(H-e\,\,|\,\,u,v\in A)\geq \wr(H-e\,\,|\,\,u\in A, v\in C).$$
\end{Lemma}

\begin{proof}
An equivalent way of phrasing the above lemma is that for a uniformly chosen random Widom-Rowlinson configurations of $H-e$ we have
$$\mathbb{P}(v\in A|u\in A)\geq \mathbb{P}(v\in C|u\in A).$$
Now let us condition not only on $u\in A$, but also on the knowledge of the set $B$. Once we know the set $B$, all connected components of $V(H-e)\setminus B$ must be monochromatic. If $u$ and $v$ are in different connected components of $V(H-e)\setminus B$ then with equal probability, the color of $v$ is red or blue. If $u$ and $v$ are in the same component of $V(H-e)\setminus B$ then with probability $1$ they have the same color so that in all cases we have
$$\mathbb{P}(v\in A|u\in A; B)\geq \mathbb{P}(v\in C|u\in A; B).$$
Hence
$$\mathbb{P}(v\in A|u\in A)\geq \mathbb{P}(v\in C|u\in A),$$
which finishes the proof of the lemma.
\end{proof}

The proof of Theorem~\ref{wr-model} is thus completed. 
\end{proof}

\begin{proof}[Proof of Corollary~\ref{sido:wr}.]
The first statement follows as before. So we only prove the second statement. 
If $H$ is connected then let $T_n$ be a spanning tree of $H$. Then, it follows from inequality~\eqref{eq:wr} that
$$\wr(H)\geq \wr(T_n)\left(\frac{7}{9}\right)^{e(H)-(n-1)}.$$
Now we use Theorem 4.3 of \cite{CL} for $G=P_3^{\circ}$, that is the following lower bound on the number of homomorphisms of a tree $T_n$ on $n$ vertices into a given graph $G$:
$$\hom(T_n,G)\geq \exp(H_{\lambda}(G))\lambda^{n-1},$$
where $\lambda$ is the largest eigenvalue of $G$, in our case it is $1+\sqrt{2}$, and $H_{\lambda}(G)$ is defined as follows:
if $\underline{y}$ is the positive eigenvector of unit length corresponding to the eigenvalue $\lambda$, then $q_i=y_i^2$ is a probability distribution $Q$ on the vertices, its entropy is 
$$H_{\lambda}(G)=\sum_{i=1}^nq_i\log \frac{1}{q_i}.$$
In our case $\underline{y}=\left(\frac{1}{2},\frac{1}{\sqrt{2}},\frac{1}{2}\right)$ so $Q=\left(\frac{1}{4},\frac{1}{2},\frac{1}{4}\right)$,
its entropy is $\frac{3}{2}\log 2$. Hence
$$\wr(T_n)=\hom(T_n,P_3^{\circ})\geq 2\sqrt{2} (1+\sqrt{2})^{n-1}.$$
Therefore, 
$$\wr(H)\geq 2\sqrt{2} (1+\sqrt{2})^{n-1}\left(\frac{7}{9}\right)^{e(H)-(n-1)}.$$
We have
$$2\sqrt{2} (1+\sqrt{2})^{n-1}\left(\frac{7}{9}\right)^{e(H)-(n-1)}=2\sqrt{2}\cdot \frac{7}{9}\frac{1}{1+\sqrt{2}}\left(\frac{3(1+\sqrt{2})}{7}\right)^n\cdot 3^n\left(\frac{7}{9}\right)^{e(H)}>$$
$$>\frac{9}{10}\left(\frac{3(1+\sqrt{2})}{7}\right)^n\cdot 3^n\left(\frac{7}{9}\right)^{e(H)}$$
and the desired inequality follows. 
\end{proof}

\section{Concluding remarks}

It is a very natural question that for which graphs $G$  the inequality
$$
\frac{\hom(H,G)}{\hom(H-e,G)}\geq \frac{\hom(K_2,G)}{v(G)^2}
$$
holds true for every (bipartite) graph $H$. It is known that one cannot handle Sidorenko's conjecture only using this tool: London \cite{Lon} showed a counterexample even when $H=P_4$, a path on $4$ vertices. We remark that in London's counterexample $G$ was a weighted graph in this case as London was interested in an inequality for matrices, not graphs. Unfortunately, one can show that it implies that there are counterexamples when $G$ is a simple graph. If the Reader is familiar with graphon theory then we can offer the following argument: consider the graphon corresponding to the weighted matrix (maybe one needs to renormalize it to make the values between $0$ and $1$), and then approximate the graphon with a large graph. If the approximation is sufficiently good then the obtained graph will also violate the above inequality. 
Of course, it only means that the above inequality cannot be true for every graph $G$, but it is still a meaningful problem to identify and study graphs for which the above inequality is true.
%\bigskip
%
%\noindent \textbf{Acknowledgment.} 
\subsection*{Acknowledgment}
We thank the anonymous referee for the helpful comments.

\end{document}